\documentclass[12pt]{amsart}
\usepackage{amssymb,latexsym}
\usepackage{enumerate}
\usepackage{float}

\usepackage{mathtools}
\usepackage{amsfonts}
\usepackage{amssymb}
\usepackage{amsmath}
\usepackage{graphicx}
\usepackage{graphicx,color}
\usepackage[normalem]{ulem}
\usepackage{todonotes}
\usepackage{stmaryrd}

\numberwithin{equation}{section}

\usepackage{amsmath,amstext,amsthm,amsxtra,mathtools,amssymb}

\usepackage[bookmarksopen,bookmarksdepth=3,colorlinks,citecolor=red,pagebackref,hypertexnames=true]{hyperref}
\usepackage[backrefs,msc-links,nobysame,initials,non-sorted-cites]{amsrefs}

\usepackage[backrefs,msc-links,nobysame,initials]{amsrefs}

\allowdisplaybreaks
\theoremstyle{theorem}
\newtheorem{thm}{Theorem}
\newtheorem{cor}{Corollary}
\newtheorem{lem}{Lemma}
\theoremstyle{definition}
\newtheorem{definition}{Definition}
\newtheorem{example}{Example}
\theoremstyle{remark}
\newtheorem{problem}{Problem}

\begin{document}

\title[Tauberian constants associated with discrete and ergodic  maximal operators]{H\"older continuity of Tauberian constants associated with discrete and ergodic strong maximal operators}
\author{Paul Hagelstein}
\address{P. H.: Department of Mathematics, Baylor University, Waco, Texas 76798}
\email{\href{mailto:paul_hagelstein@baylor.edu}{paul\!\hspace{.018in}\_\,hagelstein@baylor.edu}}
\thanks{P. H. is partially supported by a grant from the Simons Foundation (\#208831 to Paul Hagelstein).}

\author{Ioannis Parissis}
\address{I.P.: Departamento de Matem\'aticas, Universidad del Pais Vasco, Aptdo. 644, 48080
Bilbao, Spain and Ikerbasque, Basque Foundation for Science, Bilbao, Spain}
\email{\href{mailto:ioannis.parissis@ehu.es}{ioannis.parissis@ehu.es}}
\thanks{I. P. is supported by grant  MTM2014-53850 of the Ministerio de Econom\'ia y Competitividad (Spain), grant IT-641-13 of the Basque Government, and IKERBASQUE}


\begin{abstract}
   This paper concerns the smoothness of Tauberian constants of maximal operators in the discrete and ergodic settings.  In particular, we define the discrete strong maximal operator $\tilde{M}_S$ on $\mathbb{Z}^n$  by
\[
 \tilde{M}_S f(m) \coloneqq  \sup_{0 \in R \subset \mathbb{R}^n}\frac{1}{\#(R \cap \mathbb{Z}^n)}\sum_{ j\in R \cap \mathbb{Z}^n}  |f(m+j)|,\qquad m\in \mathbb{Z}^n,
\]
where the supremum is taken over all open rectangles in $\mathbb{R}^n$ containing the origin whose sides are parallel to the coordinate axes. We show that the associated Tauberian constant $\tilde{C}_S(\alpha)$, defined by
\[
\tilde{C}_S(\alpha) \coloneqq \sup_{\substack{E \subset \mathbb{Z}^n \\ 0 < \#E < \infty} } \frac{1}{\#E}\#\{m \in \mathbb{Z}^n:\, \tilde{M}_S\chi_E(m) > \alpha\},
\]
is H\"older continuous of order $1/n$.   Moreover, letting $U_1, \ldots, U_n$ denote a non-periodic collection of commuting invertible transformations on the non-atomic probability space $(\Omega, \Sigma, \mu)$ we define the associated maximal operator $M_S^\ast$  by
\[
M^\ast_{S}f(\omega) \coloneqq \sup_{0 \in R \subset \mathbb{R}^n}\frac{1}{\#(R \cap \mathbb{Z}^n)}\sum_{(j_1, \ldots, j_n)\in R}|f(U_1^{j_1}\cdots U_n^{j_n}\omega)|,\qquad \omega\in\Omega.
\]
Then the corresponding Tauberian constant $C^\ast_S(\alpha)$, defined by
\[
C^\ast_S(\alpha) \coloneqq \sup_{\substack{E \subset \Omega \\ \mu(E) > 0}} \frac{1}{\mu(E)}\mu(\{\omega \in \Omega :\, M^\ast_S\chi_E(\omega) > \alpha\}),
\]
also satisfies $C^\ast_S \in C^{1/n}(0,1).$ We will also see that, in the case $n=1$, that is  in the case of a single invertible, measure preserving transformation,  the smoothness of the corresponding Tauberian constant is characterized by the operator enabling arbitrarily long orbits of sets of positive measure.
\end{abstract}

\maketitle

\section{Introduction}
This paper is concerned with the issue of smoothness of Tauberian constants associated with discrete and ergodic maximal operators.  Tauberian constants appear at the infancy of the theory of geometric maximal operators.  Given a collection $\mathcal{B}$ of sets of finite measure in $\mathbb{R}^n$, we may define the associated maximal operator $M_\mathcal{B}$ by
\[
M_{\mathcal{B}}f(x) \coloneqq \sup_{x \in R \in \mathcal{B}}\frac{1}{|R|}\int_R |f|.
\]
In the following we will also consider discrete versions of these operators. In order to avoid ambiguities we will always assume our averaging sets in $\mathcal B$ to be \emph{open}. The same results however hold if we assume the sets in $\mathcal B$ to be closed and the proofs would also be the same.

Associated with the maximal operator $M_\mathcal{B}$ and number $\alpha\in(0,1)$ is the \emph{Tauberian constant} $C_{\mathcal{B}}(\alpha)$ defined by
\[
C_{\mathcal{B}}(\alpha) \coloneqq \sup_{\substack{E \subset \mathbb{R}^n \\ 0 < |E| < \infty}}\frac{1}{|E|}|\{x \in \mathbb{R}^n:\, M_{\mathcal{B}}\chi_E(x) > \alpha\}|.
\]
A classical result of Busemann and Feller \cite{busemannfeller1934} is that a homothecy invariant basis $\mathcal{B}$ is a density basis if and only if $C_{\mathcal{B}}(\alpha)<+\infty$  for every $0 < \alpha < 1$.   (Recall that $\mathcal{B}$ is a density basis if and only if for every measurable set $E$ we have that, for a.e. $x \in \mathbb{R}^n$,
\[
\lim_{j \rightarrow \infty}\frac{1}{|R_j|}\int_{R_j}\chi_E = \chi_E(x)
\]
holds for every sequence of sets $\{R_j\}_j$ in $\mathcal{B}$ containing $x$ whose diameters are tending to 0. See \cite{Gu} for more details.)  This result alone justifies the importance of Tauberian constants.  Furthermore, A. C\'ordoba and R. Fefferman have  shown in \cite{CorF} that Tauberian constants play a useful role in identifying classes of multiplier operators that are bounded on $L^p(\mathbb{R}^2)$ while important connections between Tauberian constants and the theory of weighted norm inequalities have been established in \cites{hlp,HP2,HPS}.

In spite of the importance of Tauberian constants in harmonic analysis and the theory of differentiation of integrals, relatively little is known about the properties of $C_{\mathcal{B}}(\alpha)$ as a function of $\alpha$ and how these properties depend on $\mathcal{B}$. Hagelstein and Stokolos proved in \cite{HS} that if $C_{\mathcal{B}}(\alpha)<+\infty$  for a \emph{single} value of $\alpha$ in $(0,1)$  then $C_{\mathcal{B}}(t)$  has at most polynomial growth in $\frac{1}{t}$ for $0 < t < 1$.  This result was extended to the weighted setting by Hagelstein, Luque, and Parissis in \cite{hlp}.

One would typically expect that, for the typical homothecy invariant density basis $\mathcal{B}$, we would have
\[
\lim_{\alpha \rightarrow 1^-}C_{\mathcal{B}}(\alpha) = 1.
\]
In general this is false, as was indicated by Beznosova and Hagelstein in \cite{bh}.  However, A. A. Solyanik proved in \cite{Solyanik} that if $\mathcal{B_{S}}$ corresponds to the collection of rectangular parallelepipeds in $\mathbb{R}^n$ whose sides are parallel to the axes, then $\lim_{\alpha \rightarrow 1^-} C_{\mathcal{B_S}}(\alpha) = 1$ with moreover the inequality $C_{\mathcal{B}_S}(\alpha) - 1 \lesssim_n (\frac{1}{\alpha} - 1)^{1/n}$ holding.  An estimate of the latter type, which quantifies how rapidly $C_{\mathcal{B}}(\alpha)$ tends to 1 as $\alpha$ tends to 1, is now referred to as a \emph{Solyanik estimate}. In \cite{HP}, Hagelstein and Parissis showed that Solyanik estimates hold when $\mathcal{B}$ is the collection of Euclidean balls in $\mathbb{R}^n$ and extended these results to the weighted setting in \cite{HP2}.

In \cite{HP2014Holder}, Hagelstein and Parissis used Solyanik estimates to prove that the Tauberian constants $C_\mathcal{B}(\alpha)$ associated with a homothecy invariant density basis of convex sets in $\mathbb{R}^n$ are locally H\"older continuous of order $p$, \emph{provided that} $C_{\mathcal{B}}(\alpha)$ satisfies a Solyanik estimate of the form $C_{\mathcal{B}}(\alpha) - 1 \lesssim (\frac{1}{\alpha} - 1)^p$.    We briefly indicate the nature of the proof in the special but important case that $\mathcal{B}$ is a homothecy invariant collection of rectangular parallelepipeds in $\mathbb{R}^n$.  For any collection $\mathcal{B}$ one may define the associated \emph{halo} $\mathcal{H}_{\mathcal{B}, \alpha}(E)$ of a measurable set $E$ with respect to $\alpha$ by
\[
\mathcal{H}_{\mathcal{B}, \alpha}(E) \coloneqq \{x \in \mathbb{R}^n:\, M_{\mathcal{B}}\chi_E(x) > \alpha\}.
\]
If $\mathcal{B}$ is a homothecy invariant collection of rectangular parallelepipeds in $\mathbb{R}^n$ one  has the iterated halo containment relation
\[
\mathcal{H}_{\mathcal{B}, \alpha}(E) \subset \mathcal{H}_{\mathcal{B}, \alpha(1 + \frac{\delta}{  2^n})}(\mathcal{H}_{\mathcal{B}, 1 - 2\delta}(E))
\]
for sufficiently small $\delta > 0$, immediately implying that
\[
C_\mathcal{B}(\alpha) \leq C_{\mathcal{B}}(\alpha(1 + \frac{\delta}{2^n}))C_{\mathcal{B}}(1 - 2\delta)
\]
for sufficiently small $\delta > 0$.   This inequality, combined with the estimate $C_{\mathcal{B}}(\alpha) - 1 \lesssim (\frac{1}{\alpha} - 1)^p$, suffices to show that $C_{\mathcal{B}}(\alpha)$ lies in the H\"older class $C^p(0,1)$.  We remark that the ideas of the above proof, combined with known Solyanik estimates for the uncentered Hardy-Littlewood maximal operator $M_{\textup{HL}}$ defined by
\[
M_{\textup{HL}}f(x) \coloneqq \sup_{x \in Q }\frac{1}{|Q|}\int_{Q} |f|,
\]
where the supremum is over all cubes with sides parallel to the coordinate axes containing $x$, may be used to show that the associated Tauberian constants $C_{\textup {HL}}(\alpha)$ for the uncentered Hardy-Littlewood maximal operator satisfy the smoothness estimate $C_{\textup{HL}}\in C^{1/n}(0,1)$.  The details of the rather delicate, associated argument may be found in \cite{HP2014Holder}.

In the recent paper \cite{HPSolErg}, Hagelstein and Parissis considered the issue of Solyanik estimates in the setting of ergodic theory. A result from \cite{HPSolErg} that we are particularly interested in here is the following.  Let $U_1, \ldots, U_n$ be a collection of invertible measure preserving transformations on a probability space $(\Omega, \Sigma, \mu)$ and define the associated strong ergodic maximal operator $M_S^*$ by
\[
M^\ast_{S}f(\omega) \coloneqq \sup_{0 \in R \subset \mathbb{R}^n}\frac{1}{\#(R \cap \mathbb{Z}^n)}\sum_{(j_1, \ldots, j_n)\in  R\cap \mathbb Z^n}|f(U_1^{j_1}\cdots U_n^{j_n}\omega)|,\qquad\omega\in\Omega,
\]
where the supremum is taken over all rectangular parallelepipeds $R$ in $\mathbb R^n$ with sides parallel to the coordinate axes that contain the origin. The corresponding Tauberian constant $C^\ast_S(\alpha)$  by
\[
C^\ast_S(\alpha) \coloneqq \sup_{\substack{E \subset \Omega \\ \mu(E) > 0}} \frac{1}{\mu(E)}\mu(\{\omega \in \Omega:\, M^\ast_S\chi_E(\omega) > \alpha\}),\qquad 0 < \alpha < 1.
\]
We have that $C^\ast _S (\alpha) $ satisfies the ergodic Solyanik estimate
\[
C^\ast _S (\alpha) - 1 \lesssim_n \Big(\frac{1}{\alpha} - 1\Big)^{1/n}.
\]
From this estimate and from an awareness of the H\"older smoothness estimates exhibited above, one might expect that $C^\ast_S (\alpha)$ should satisfy a H\"older smoothness estimate on $(0,1)$.  These expectations are dashed by the following example, arising even in the case $n=1$.

\begin{example}
Define $T$ on $[0,1)$ equipped with the Lebesgue measure by
\[
T(x) \coloneqq \Big(x + \frac{1}{2}\Big) \bmod 1.
\]
Setting
\[
T^*f(\omega) \coloneqq \sup_{M \leq 0 \leq N \in \mathbb{Z}} \frac{1}{N - M + 1}\sum_{j=M}^{N}|f(T^j\omega)|
\]
and the corresponding  Tauberian constant $C^*_T(\alpha)$ by
\[
C^\ast_T(\alpha) \coloneqq \sup_{\substack{E \subset \Omega \\ \mu(E) > 0}} \frac{1}{\mu(E)}\mu(\{\omega \in \Omega :\, T^*\chi_E(\omega) > \alpha\}),
\]
we have that the associated Tauberian constants $C^\ast_T (\alpha)$ satisfy the formula
\[
C_T^\ast(\alpha) =
\begin{cases}
  2 & \text{if}\ 0 < \alpha < \frac{2}{3} \\
  1          & \text{if}\ \frac{2}{3} \leq \alpha < 1.
\end{cases}
\]
To see this, note that if $E \subset [0,1)$, $T^\ast \chi_E$ only takes on the values $0$, $\frac{2}{3}$, or 1.  If $x \in E$, then of course $T^\ast \chi_E(x) = 1$.   If $x$ and $Tx$ are not in $E$, then $T^\ast \chi_E(x) = 0$.  If $x \notin E$ but $Tx \in E$, then $T^\ast \chi_E(x) = \frac{2}{3}$.  These observations together with the possibility of the set $E$ being, say, $[0, \frac{1}{3}]$ yield the above formula.
\end{example}

If $T$ is an \emph{ergodic} transformation on a non-atomic probability space or even if $T$ is just \emph{non-periodic}, then $C^\ast _T (\alpha)$ is smooth on $(0,1)$, and in fact $C^\ast _T (\alpha) = \frac{2}{\alpha} - 1$.   This result, explicitly proven later in the paper, follows relatively easily from transference principles and a sharp Tauberian estimate for the uncentered Hardy-Littlewood maximal operator on $\mathbb{R}$ due to Solyanik. We also have that smoothness estimates hold for Tauberian constants associated  with the strong ergodic maximal operator  given by a non-periodic collection of commuting invertible measure preserving transformations on a probability space. This is the primary result of this paper, formally stated as follows.

 \begin{thm}\label{t.main}
 Let $n \geq 2$ and $\{U_1, \ldots, U_n\}$ be a non-periodic collection of commuting invertible measure preserving transformations on a   probability space $(\Omega, \Sigma, \mu)$.   Then the Tauberian constants $C_S^\ast(\alpha)$ of the associated strong ergodic maximal operator $M_S^\ast$ lie in the H\"older class $C^{1/n}(0,1)$.  Moreover, corresponding to the $n=1$ case, if $C_T^\ast(\alpha)$ is the Tauberian constant with respect to $\alpha$ associated with a non-periodic invertible measure preserving transformation $T$ on $(\Omega, \Sigma, \mu)$, then ${C}_T^\ast(\alpha)$ is given by the formula
\[
{C}_T^\ast(\alpha) = \frac{2}{\alpha} -1, \qquad\alpha \in (0,1),
\]
and is thus smooth on $(0,1)$.
\end{thm}

Before we get to the details of the proof in subsequent sections, a few words regarding this proof are in order.  One might suspect that, given the Solyanik estimates already at hand for the maximal operator $M_S^\ast$, one could prove a containment relation along the lines of
 \[
\mathcal{H}^\ast_{S, \alpha}(E) \subset \mathcal{H}^\ast_{S, \alpha(1 + \frac{\delta}{  2^n})}(\mathcal{H}^\ast_{S, 1 - 2\delta}(E)),
\]
where $\mathcal{H}^*_{S, \alpha}(E)$ is the halo set associated with the maximal operator $M_S^\ast$ acting on functions on a probability space. \emph{Such a containment relation is in general false.} For example, consider the $n=1$ case and let $\Omega = [0,1)$ be equipped with Lebesgue measure.  Let $U_1(x) = (x + \frac{1}{2}) \bmod 1$.   Setting $\alpha = 0.49$ and $\delta = 0.1$, we have $\mathcal{H}^\ast_{S, \alpha}([0,1/2)) = [0,1)$ but $ \mathcal{H}^\ast_{S, \alpha(1 + \frac{\delta}{2})}(\mathcal{H}^\ast_{S, 1 - 2\delta}([0,1/2))) = \mathcal{H}^\ast_{S, \alpha(1 + \frac{\delta}{2})}([0,1/2)) = [0, 1/2).$ Underlining this example is the realization that, given a set $E$ on a probability space, for small $\delta > 0$ the halo $\mathcal{H}^\ast_{S, 1 - \delta}(E)$ of $E$ could very well be the set $E$ itself, a scenario that does not happen in the typical geometric setting in which halos of sets are quantifiably larger than the sets themselves.

This lack of halo containment also manifests itself in the context of the \emph{discrete strong maximal operator on $\mathbb{Z}^n$}, denoted by here $\tilde{M}_S$ and defined by
\[
 \tilde{M}_S f(m) \coloneqq  \sup_{0 \in R \subset \mathbb{R}^n}\frac{1}{\#(R \cap \mathbb{Z}^n)}\sum_{ j \in R \cap \mathbb{Z}^n}  |f(m+j)|,\qquad m\in \mathbb{Z}^n,
\]
where the supremum is taken over all open rectangles in $\mathbb{R}^n$ containing the origin whose sides are parallel to the coordinate axes.
We may define the associated halo function $\tilde{\mathcal{H}}_{S,\alpha}(E)$ by
\[
\tilde{\mathcal{H}}_{S,\alpha}(E) \coloneqq \{n \in \mathbb{Z}^n : \, 	\tilde{M}_S\chi_E(n) > \alpha\}.
\]
Observe that for small $\delta > 0$ we might have $\tilde{\mathcal{H}}_{S, 1 - \delta}(E)$ is equal to $E$ itself, as for instance in the simple case that $E = \{0\}$.

Coming to the rescue, the desired halo containment \emph{is} satisfied by the  \emph{continuous} strong maximal operator $M_S$ on $\mathbb{R}^n$, defined by
\[
M_Sf(x) \coloneqq \sup_{x \in R}\frac{1}{|R|}\int_R |f|,
\]
the supremum being taken over all open rectangles in $\mathbb{R}^n$ containing $x$ whose sides are parallel to the coordinate axes.
It is in fact this halo containment, combined with  Solyanik estimates for $M_{S}$, that enables a proof of the Lipschitz continuity of the Tauberian constants $C_{S}(\alpha)$ associated  with $M_{S}$.  In this paper we will see that the Tauberian constants associated  with $M_S^\ast$ and $M_S$ are equal provided $M_S^\ast$ is associated with a non-periodic collection $U_1, \ldots, U_n$ of commuting invertible measure preserving transformations on a non-atomic probability space. These considerations  prove the desired H\"older continuity for $C^\ast_S(\alpha)$.

Ideas in the above proof may also be used to show that the Tauberian constants associated with the ``one sided'' ergodic maximal operator $T^{\ast+}$, defined by
\[
T^{*+}f(\omega) \coloneqq \sup_{N \geq 0} \frac{1}{N + 1}\sum_{j=0}^{N}|f(T^j\omega)|,
\]
as well as the ``two sided'' ergodic maximal operator $T^\ast f$, defined earlier, are  Lipschitz continuous and in fact smooth on $(0,1)$.  The proofs of these result are significantly easier than those in the multiparameter scenario, but we highlight them as they  relate to the maximal operators most prevalent in ergodic theory.
\\

A few words regarding the organization of the remainder of the paper are in order.   In Section~\ref{s.not} we provide some explanatory comments regarding notation used in the paper. In Section~\ref{s.HolderCS} we will prove that the Tauberian constants associated with $M_{S}$ and $\tilde{M}_S$ are the same, consequently ascertaining that the Tauberian constants associated with $\tilde{M_S}$ are H\"older continuous if $n \geq 2$ and smooth if $n=1$.   In Section~\ref{s.HolderCS*} we will prove that, if $U_1, \ldots, U_N$ are a  non-periodic collection of commuting invertible measure preserving transformations on a non-atomic probability space, then the Tauberian constants associated with $M_S^\ast$ and $\tilde{M_S}$ are equal.  The latter proof will use, not unexpectedly, the Calder\'on transference principal as well as a Kakutani-Rokhlin type theorem due to Katznelson and Weiss, \cite{KW}. As a corollary we will obtain the desired result that the Tauberian constants $C_{S}^\ast(\alpha)$ are H\"older continuous on $(0,1)$ and that, in the $n=1$ case, the associated $C_T^\ast(\alpha)$ is smooth. In Section~\ref{s.onesided} we will provide a proof that the function $C_{T}^{*+}(\alpha)$ associated  with the Tauberian constants of the one-sided ergodic maximal operator $T^{*+}$ associated  with a non-periodic transformation $T$ is smooth.  In the last section, \S\ref{s.problems}, we will indicate some open problems and suggested directions of further research.

\section{Notation}\label{s.not}  We write $A\lesssim_\tau B$ whenever $A\leq C_\tau B$ for some numerical constant $C_\tau>0$ depending on some parameter $\tau$. Then $A\simeq B $ whenever $A\lesssim B$ and $B\lesssim A$. Throughout the paper $(\Omega,\Sigma,\mu)$ is a probability space and $T$ will be an invertible measure preserving transformation on $(\Omega,\Sigma,\mu)$, which might or might not be ergodic. For a set $E\subset \mathbb Z$ we denote by $\#E$ the cardinality of $E$. We many times use the multi-index notation $m=(m_1,\ldots,m_n)\in\mathbb Z^n$ for points in the integer lattice $\mathbb Z^n$. Finally, $\chi_E$ denotes the indicator function of a measurable set $E$, either in $\mathbb{R}^n$ or $\Omega$, depending on context.

\section{H\"older continuity of $\tilde{C}_S(\alpha)$}\label{s.HolderCS}
In this section, we show that the Tauberian constants $\tilde{C}_S(\alpha)$ associated  with the discrete strong maximal operator $\tilde{M}_S$ are H\"older continuous on $(0,1)$ and in fact smooth when $n=1$.

\begin{lem}
For $0 < \alpha < 1$, let  $\tilde{C}_S(\alpha)$ and $C_{S}(\alpha)$  denote the Tauberian constants with respect to $\alpha$, associated with the discrete strong maximal operator $\tilde{M}_S$, and the continuous strong maximal operator  $M_{S}$, respectively.
Then
\[
\tilde{C}_S(\alpha) = C_{S}(\alpha).
\]
\end{lem}

\begin{proof}
We first show that $\tilde{C}_S(\alpha) \leq C_{S}(\alpha)$.  Let $\tilde{E}$ be a finite set in $\mathbb{Z}^n$.  We associate to $\tilde{E}$ a set $E \subset \mathbb{R}^n$ defined by
\[
\chi_E(x_1, \ldots, x_n) \coloneqq \chi_{\tilde{E}}(\lfloor x_1 \rfloor , \lfloor x_2 \rfloor , \ldots, \lfloor x_n \rfloor);
\]
here and in the rest of the paper, for $x\in\mathbb R$ we denote by $\lfloor x \rfloor$ the largest integer which is less than or equal to $x$. For $j=(j_1,\ldots,j_n)\in\mathbb Z^n$ we write
\[
\Phi_j\coloneqq [j_1,j_1+1)\times \cdots\times [j_n,j_n+1).
\]
With this notation we have that $|E|=\sum_{j\in \tilde E} |\Phi_j|= \# \tilde E$.

For any axis parallel rectangular parallelepiped $R\ni 0$ in $\mathbb R^n$ and $m=(m_1,\ldots ,m_n)\in \mathbb Z^n$ we now have the identity
\[
\begin{split}
	\frac{1}{\# (R\cap \mathbb Z^n)} \sum_{j\in R\cap \mathbb Z^n} \chi_{\tilde E}(m+j)&=\frac{1}{\# (R\cap \mathbb Z^n)} \sum_{j\in R\cap \mathbb Z^n}\int_{\Phi_{m+j}} \chi_{\tilde E}(\lfloor u_1 \rfloor,\ldots,\lfloor u_n \rfloor)du
	\\
	& = \frac{1}{\# (R\cap \mathbb Z^n)} \int_{S_{m,R}} \chi_ E (u) du
\end{split}
\]
where we have defined $S_{m,R}\coloneqq \cup_{j\in R\cap \mathbb Z^n} \Phi_{m+j	}$, where we remember that $R$ is taken to be open. Observe that $S_{m,R}$ is an axis parallel rectangular parallelepiped in $\mathbb R^n$ with $|S_{m,R}|=\# (R\cap \mathbb Z^n)$ and that $S_{m,R}\supseteq \Phi_m$ since $R\ni 0$. We conclude that for any axis parallel rectangular parallelepiped $R$ in $\mathbb R^n$ and any $m\in\mathbb Z^n$ we have that $M_{S}\chi_E(x)\geq \tilde M_{S}\chi_{\tilde E}(m)$ for $x\in \Phi_m$. As $\#\tilde E=|E|$ we conclude that $\tilde{C}_S(\alpha) \leq C_{S}(\alpha)$ as we wanted.

We now show the more interesting inequality $\tilde{C}_S(\alpha) \geq C_{S}(\alpha)$. Let $\alpha\in(0,1)$ and $\epsilon>0$ be fixed throughout the proof. We remember here that $1\leq C_S(\alpha)<+\infty$, using for example the $L^p$ bounds for the strong maximal function and the strong differentiation theorem. Now we choose a measurable set $E$ with $0<|E|<+\infty$, and such that
\[
|\{x\in\mathbb R^n:\, M_S \chi_E(x)>\alpha|>(C_S(\alpha)-\epsilon)|E|
\]
By the outer regularity of the Lebesgue measure there exists an open set $U\supseteq E$ such that $|U\setminus E|<\epsilon |E|/C_s(\alpha)$ from which we get
\[
|\{x\in\mathbb R^n:\, M_S \chi_U(x)>\alpha|>(C_S(\alpha)-2\epsilon)|U|.
\]
Now $U$ is open so it can be written as a countable union of dyadic cubes $\{\tilde Q_k\}_k$ with disjoint interiors. By Fatou's lemma there exists a finite subcollection $\{Q_j\}_{j=1} ^N \subseteq \{\tilde Q_k\}_k$, such that
\[
\begin{split}
|\{x\in\mathbb R^n:\, M_S \chi_{\cup_{j=1} ^N Q_j}(x)>\alpha\}| & \geq  |\{x\in\mathbb R^n:\, M_S \chi_U(x)>\alpha\}| -\epsilon|U|
\\
&>(C_S(\alpha)-3\epsilon) |U|
\\
&\geq (C_S(\alpha)-3\epsilon)\Big|\bigcup_{j=1} ^N Q_j\Big|
\end{split}
\]
Since the collection $\{Q_j\}_{j=1} ^N$ is a finite collection of dyadic cubes we can assume that all the cubes in the collection have the same side-length, by splitting, if necessary, the larger cubes finitely many times.

We have showed that for the given $\alpha\in(0,1)$ and $\epsilon>0$ there exists a finite collection of dyadic cubes $\{Q_j\}_{j=1} ^N$, with equal side-length and disjoint interiors, such that
\[
|\{x\in\mathbb R^n:\, M_S \chi_{\cup_{j=1} ^N Q_j}(x)>\alpha\}|>(C_S(\alpha)-3\epsilon)\Big|\bigcup_{j=1} ^N Q_j\Big|
\]
Now by definition there exists a collection of rectangles $\mathcal R$ such that $|\tilde R\cap {\cup_{j=1} ^N Q_j}|>\alpha |\tilde R|$ for every $\tilde R\in\mathcal R$ and
\[
\{x\in\mathbb R^n:\, M_S \chi_{\cup_{j=1} ^N Q_j}(x)>\alpha\} = \bigcup_{R\in\mathcal R} R.
\]
By the inner regularity of the Lebesgue measure we can then find a finite subcollection $\{\tilde R_\tau\}_{\tau=1} ^M\subseteq \mathcal R$ such that
\[
\Big|\bigcup_{\tau=1} ^M \tilde R_\tau\Big| > (C_S(\alpha)-4\epsilon)\Big|\bigcup_{j=1} ^N Q_j\Big|
\]
and of course for each $\tau$ we have $|\tilde R_\tau\cap {\cup_{j=1} ^N Q_j}|>\alpha |\tilde R_\tau|$. Then for each $\tau$ there exists $\delta_\tau>0$ such that $|\tilde R_\tau\cap {\cup_{j=1} ^N Q_j}|>(\alpha+\delta_\tau) |\tilde R_\tau|$. Now for each $\tau$ we choose a rectangle $ R_\tau \supseteq \tilde R_\tau$, where $R_\tau$ has corners with rational coordinates, and
\[
| R_\tau\setminus \tilde  R_\tau|< \frac{\delta_\tau}{\alpha+\delta_\tau}|\tilde R_\tau| .
\]
Then we still have $|\cup_{\tau=1} ^M R_\tau|>(C_S(\alpha)-4\epsilon)|\cup_{j=1} ^N Q_j|$ and for each $\tau\in\{1,\ldots,M\}$
\[
\Big| R_\tau \cap \bigcup_{j=1} ^N Q_j\Big|\geq \frac{\alpha+\delta_\tau}{\alpha} \frac{|\tilde R_\tau|}{| R_\tau|}\alpha| R_\tau|>\alpha | R_\tau|.
\]
As the rectangles in the collection $\{R_\tau\}_{\tau=1} ^M$ have rational corners and they are finitely many, we can use the dilation invariance of the operator $M_S$ to rescale everything so that all the cubes in $\{Q_j\}_{j=1} ^N$ and all the rectangles in $\{R_\tau\}_{\tau=1} ^M$ have corners on the integer lattice $\mathbb Z^n$ and the cubes in $\{Q_j\}_{j=1} ^N$ still have equal side-lengths and disjoint interiors.

Now let us define the set $\tilde E$ to consist of all the \emph{lower left} corners of the cubes in $\{Q_j\}_{j=1} ^N$ and for a rectangle $R= (a_1,b_1)\times \cdots  (a_n,b_n)$ define the larger rectangle $ R^\prime \coloneqq (a_1-1,b_1)\times \cdots \times  (a_n-1,b_n)$. Then for all $\tau$ we have
\[
\alpha \#(  R_\tau ^\prime \cap \mathbb Z^n)=\alpha|R _\tau |<\Big| R _\tau \cap \bigcup_{j=1} ^N Q_j\Big| =\sum_{j:\, Q_j \subseteq R_\tau} |Q_j| \leq \# (\tilde E \cap  R_\tau ^\prime).
\]
This shows that $\tilde M_S \chi_{\tilde E } (m)>\alpha$ for all $m\in \cup_{\tau=1} ^M R^\prime _\tau\cap \mathbb Z^n$. Thus
\[
\begin{split}
\#\{m\in\mathbb Z^n:\,   \tilde M_S \chi_{\tilde E } (m)>\alpha\}& >\#\Big(\bigcup_{\tau=1} ^M R^\prime _\tau\cap \mathbb Z^n\Big)
\\
& = \Big| \bigcup_{\tau=1} ^M R_\tau \Big| >(C_S(\alpha)-4\epsilon) \Big|\bigcup_{j=1} ^N Q_j\Big|
\\
&=(C_S(\alpha)-4\epsilon) \# \tilde E
\end{split}
\]
as every point of $\tilde E$ corresponds to exactly one of the cubes $Q_j$. As the left hand side is independent of $\epsilon>0$ and $\epsilon$ was arbitrary, this completes the proof.
\end{proof}

\begin{cor}\label{c.smoothness}
For $0 < \alpha < 1$, let $\tilde{C}_S(\alpha)$ denote the associated Tauberian constant of the discrete strong maximal operator $\tilde{M}_S$ acting on functions on $\mathbb{Z}^n$.  Then
\begin{enumerate}
	\item [(i)] In dimensions $n\geq 2$ we have $\tilde{C}_S  \in C^{1/n}(0,1)$.
	\item[(ii)] In dimension $n=1$ we have $\tilde{C}_S \in C^\infty(0,1)$ satisfying the equation $\tilde{C}_S(\alpha) = \frac{2}{\alpha} - 1$ for all $\alpha\in(0,1)$.
\end{enumerate}
\end{cor}

\begin{proof}
From \cite{HP2014Holder}*{Corollary 2} we have that $C_{S}(\alpha) \in C^{1/n}(0,1)$. Moreover, from the main theorem of \cite{Solyanik} we have that, for $n=1$, $C_{S}(\alpha) = \frac{2}{\alpha} - 1$.   (Note that in the $n=1$ case the strong maximal operator is the same as the uncentered Hardy-Littlewood maximal operator.)   By the above lemma the desired result holds.
\end{proof}

\section{H\"older continuity of $C_S^\ast(\alpha)$}\label{s.HolderCS*}
We now show that, if $U_1, \ldots, U_n$ form a non-periodic collection of commuting  invertible transformations on the non-atomic probability space $(\Omega, \Sigma, \mu)$, for every $0 < \alpha < 1$ the associated Tauberian constants $C^\ast_S(\alpha)$ and $\tilde{C}_S(\alpha)$ are the same. The fact that $C^\ast_S(\alpha) \leq \tilde{C}_S(\alpha)$ follows readily using the Calder\'on transference principle.

\begin{lem}\label{l.ergfromdisc}
Let $U_1, \ldots, U_n$ form a collection of commuting invertible measure preserving transformations on  a  probability space $(\Omega, \Sigma, \mu)$  and for $\alpha\in(0,1)$ let $C^\ast_S(\alpha)$ and $\tilde{C}_S(\alpha)$ denote the Tauberian constants associated  with the maximal operators $M_S^\ast$ and $\tilde{M}_S$, respectively.  Then
\[
C^\ast_S(\alpha) \leq \tilde{C}_S(\alpha).
\]
\end{lem}

\begin{proof}
This result follows immediately from  \cite{HPSolErg}*{Theorem 3.1}, where $\mathcal{B}$ is taken to be the collection of all open rectangular parallelepipeds in $\mathbb{R}^n$ containing the origin and whose sides are parallel to the coordinate axes.
\end{proof}

The inequality $\tilde{C}_S(\alpha) \leq C^\ast_S(\alpha)$ does \emph{not} hold in general, as can be seen even in the $n=1$ case by setting $U_1(x) \coloneqq (x + \frac{1}{2}) \bmod 1$ on the probability space $[0,1)$ equipped with the Lebesgue measure.  For this transformation  we have $1 = C^\ast_S(\frac{2}{3}) < \tilde{C}_S(\frac{2}{3})$, the latter being at least 2.   However, we shall see that if $U_1, \ldots, U_n$ form a non-periodic collection of commuting invertible measure preserving transformations on a nonatomic probability space $(\Omega, \Sigma, \mu)$, we do have that $\tilde{C}_S(\alpha) \leq C^\ast_S(\alpha)$ and hence equality between $\tilde{C}_S(\alpha)$ and $C^\ast_S(\alpha)$ holds.

\begin{definition} Let $T$ be an invertible measure preserving transformation on a probability space $(\Omega,\Sigma,\mu)$.	A point $\omega\in \Omega$ is \emph{a periodic point} if there exists a positive integer $n$ such that $T^n \omega=\omega$. Alternatively we say that \emph{$T$ is periodic at $\omega$.} The transformation $T$ is called \emph{non-periodic} if the set of its periodic points has measure zero, that is, if it is almost nowhere periodic.   More generally, a collection of commuting invertible measure preserving transformations $U_1, \ldots, U_n$ is said to be non-periodic if for every $(l_1, \ldots, l_n) \in \mathbb{Z}^n\backslash \{0\}$ we have
\[
\mu\{x \in \Omega :\, U_{1}^{\ell_1}\cdots U_{n}^{\ell_n} x = x\} = 0.
\]
\end{definition}
With this definition in hand we can now show that the Tauberian constants of an ergodic strong maximal operator associated with a non-periodic collection of invertible measure preserving transformations coincide with those of the discrete strong maximal operator.
\begin{lem}\label{l.equal}
Let $U_1, \ldots, U_n$ form a non-periodic collection of commuting invertible measure preserving transformations on a probability space $(\Omega, \Sigma, \mu)$ and let, for $0 < \alpha < 1$, $C^\ast_S(\alpha)$ and $\tilde{C}_S(\alpha)$ be the Tauberian constants associated  with the maximal operators $M_S^\ast$ and $\tilde{M}_S$.   Then
\[
C^\ast_S(\alpha) = \tilde{C}_S(\alpha).
\]
\end{lem}

\begin{proof}
Let $0 < \alpha < 1$. By Lemma~\ref{l.ergfromdisc} it suffices to show that $  C^\ast_S(\alpha) \geq \tilde{C}_S(\alpha)$.  Let $\tilde{E}$ be a nonempty set in $\mathbb{Z}^n$ with finitely many points and $N \in \mathbb{Z}_+$ be such that $\{m \in \mathbb{Z}^n : \, \tilde{M}_S\chi_{\tilde{E}}(m) > \alpha\} \subseteq [-N, N]^n$. By a refinement of the Kakutani-Rokhlin lemma, due to Katznelson and Weiss \cite{KW}, there exists a set $A \subset \Omega$ of positive measure  such that $U_1^{j_1}U_2^{j_2}\cdots U_n^{j_n}A$ are pairwise disjoint where $0 \leq j_i \leq N$ for $i = 1, \ldots, n$.   Define the set $E$ in $\Omega$ by
\[
E \coloneqq \bigcup_{(j_1, \ldots, j_n) \in \tilde{E}} U_1^{j_1}U_2^{j_2}\cdots U_n^{j_n} A\eqqcolon \bigcup_{j\in \tilde E}E_j.
\]
Now we claim that
\[
\bigcup_{\substack{m\in\mathbb Z^n:\, \tilde{M}_S\chi_{\tilde{E}}(m) > \alpha}} U_1 ^{m_1}\cdots U_n ^{m_n}A \subseteq \{\omega\in \Omega:\, M_S^*\chi_{E}(\omega)>\alpha\}.
\]
Indeed, let $m\in\mathbb Z^n$ and $R$ a rectangular parallelepiped in $\mathbb R^n$ with $0\in R$ such that
\[
\frac{1}{\#(R\cap \mathbb Z^n)}\sum_{k \in R\cap \mathbb Z^n} \chi_{\tilde E}(k+m)>\alpha.
\]
Then by the disjointess of the sets $\{E_j\}_j$ we have for $\omega \in U^{m_1}\cdots U^{m_n}A$ that
\[
\begin{split}
\frac{1}{\#(R\cap \mathbb Z^n)}\sum_{j\in R\cap \mathbb Z^n} \chi_E (U^{j_1}\cdots U^{j_n}\omega)& = \frac{1}{\#(R\cap \mathbb Z^n)}\#\{j\in R\cap \mathbb Z^n:\, U^{j_1}\cdots U^{j_n} \omega \in \cup_{k\in \tilde E} E_k\}
\\
&=  \frac{1}{\#(R\cap \mathbb Z^n)} \sum_{k\in \tilde E}  \#\{j\in R\cap \mathbb Z^n:\, U^{j_1}\cdots U^{j_n} \omega \in E_k\}
\\
&\geq \frac{1}{\#(R\cap \mathbb Z^n)} \sum_{k\in \tilde E}\#\{j\in R\cap \mathbb Z^n:\, m+j=k\}
\\
& =  \frac{1}{\#(R\cap \mathbb Z^n)} \#\bigcup_{k\in \tilde E} \{j\in R\cap \mathbb Z^n:\, m+j=k\}
\\
& = \frac{1}{\#(R\cap \mathbb Z^n)} \sum_{j\in R\cap \mathbb Z^n} \chi_{\tilde E}(m+j)>\alpha.
\end{split}
\]
Accordingly,
\[
\begin{split}
\mu(\{\omega \in \Omega :\, M_S^*\chi_{E}(\omega) > \alpha\}) &\geq \sum_{m \in \mathbb{Z}^n : \, \tilde{M}_S\chi_{\tilde{E}}(m) > \alpha  } \mu(U^{m_1}\cdots U^{m_n}A).
\\
& \geq \mu(A)\#\{m \in \mathbb{Z}^n : \, \tilde{M}_S\chi_{\tilde{E}}(m) > \alpha  \}.
\end{split}
\]

Since $\mu(E) = \mu(A) \#\tilde{E}$, it follows that
\[
\frac{\mu(\{\omega \in \Omega :\, M_S^*\chi_{E}(\omega) > \alpha\})}{\mu(E)} \geq \frac{\#\{m \in \mathbb{Z}^n : \tilde{M}_S\chi_{\tilde{E}}(m) > \alpha \}}{\#\tilde{E}}.
\]
As $\tilde{E}$ was arbitrary in $\mathbb{Z}^n$, we get $C^\ast_S (\alpha) \geq \tilde{C}_S(\alpha)$ as desired.
\end{proof}

\begin{proof}[Proof of Theorem~\ref{t.main}]
The proof  follows immediately from Corollary~\ref{c.smoothness} and Lemma~\ref{l.equal}.
\end{proof}

\subsection{A characterization of smoothness for the Tauberian constant of a single measure preserving transformation}
In the case of a single invertible measure preserving transformation we can actually state and prove a characterization of smoothness of $C_S ^*$. For this we introduce the following index of an invertible, measure preserving transformation $T$ on a probability space $(\Omega,\Sigma,\mu)$.

\begin{definition} Let $T$ be an invertible measure preserving transformation acting on a probability space $(\Omega,\Sigma,\mu)$. If for every positive integer $N$ there exists a measurable set $A\subset \Omega$ with $\mu(A)>0$, such that the sets $A,TA,\ldots,T^{N}A$ are disjoint we define the index of $T$ to be $N_T\coloneqq \infty$. Otherwise the index of $T$ is to defined to be the largest positive integer $N_T$ for which there exists a measurable set $A\subset \Omega$ with $\mu(A)>0$ such that the sets $A,T A,\ldots,T^{N_T-1}A$  are pairwise disjoint.
\end{definition}
Note that if $T$ is non-periodic the Kakutani-Rokhlin lemma implies that $N_T=\infty$. However the condition $N_T=\infty$ is in general strictly weaker than the non-periodicity condition in the assumption of the Kakutani-Rokhlin lemma. Indeed, consider for example $T_1:[0,1/2)\to [0,1/2)$, equipped with the Lebesgue measure, to be (say) ergodic and $T_2:[1/2,1)\to [1/2,1]$ to be the identity. Then $T\coloneqq T_1\oplus T_2 :[0,1]\to [0,1]$ inherits the property $N_T=\infty$ from $T_1$ but it is obvious that $T$ fails the non-periodicity assumption because of $T_2$.

The case $N_T=1$ is of special importance as, in this case, we can calculate exactly $C_T ^*(\alpha)$.
\begin{lem}\label{l.index1}Let $T$ be an invertible measure preserving transformation on a probability space $(\Omega,\Sigma,\mu)$ and suppose that $T$ has index $N_T=1$. Then $C^* _T(\alpha)=1$ for all $\alpha\in[0,1)$.
\end{lem}

\begin{proof} Let $A\subset \Omega$ be a set of positive measure. Since we obviously have $M_S \chi_A(\omega) =1$ for every $\omega\in A$ it will be enough to show that we also have $M_S  \chi_A(\omega)  = 0$ for $\mu$-a.e. $\omega\in\Omega\setminus A$. To do this, it suffices to show that, if $0 < \mu(A) < 1$, then for $\mu$-a.e. $\omega \in \Omega \setminus A$ we have $T\omega \in \Omega \setminus A$ and that $T^{-1}\omega \in \Omega \setminus A$ (the both of which would imply that for $\mu$-a.e. $\omega \in \Omega \setminus A$ that $T^j \omega \in \Omega \setminus A$ for every $j$.)  Well, if the assertion $ T\omega \in \Omega \setminus A$ for $\mu$-a.e. $\omega \in \Omega \setminus A$ were false, then there would be a set $\tilde{A} \subset \Omega \setminus A$ with $\mu(\tilde{A}) > 0$ such that $T \tilde{A} \subset A$.   But as $\tilde{A}$ and $A$ are disjoint and $T$ is an invertible measure preserving transformation, we would have that $\tilde{A}$ and $T\tilde{A}$ constitute disjoint sets of positive $\mu$ measure, contradicting the assumption that $N_T = 1$.  If the assertion $ T^{-1}\omega \in \Omega \setminus A$ for $\mu$-a.e. $\omega \in \Omega \setminus A$  were false, then there would be a set $\tilde{A} \subset \Omega \setminus A$  of positive measure such that $T^{-1}\tilde{A} \subset A$.  But then  $T^{-1}\tilde{A}$, $T(T^{-1}\tilde{A}) = \tilde{A}$ would constitute disjoint sets of positive measure, again contradicting the assumption that $N_T = 1$.
\end{proof}
We can now give a characterization of smoothness for $C_T ^*$ in terms of the index $N_T$.

\begin{thm}\label{t.index}Let $T$ be an invertible measure preserving transformation on a probability space $(\Omega,\Sigma,\mu)$ with index $N_T\in[1,\infty]$. Then there are the following possibilities
	\begin{itemize}
		\item [(i)] If $N_T=1$ then $C_T ^*(\alpha)=1$ on $[0,1)$ and thus $C_T ^*\in C^\infty(0,1)$.
		 \item [(ii)] If $N_T=\infty$ then $C_T ^*(\alpha)=\frac{2}{\alpha}-1$ on $(0,1)$ and thus $C_T ^*\in C^\infty(0,1)$.
		 \item[(iii)] If $1<N_T<\infty$ then $C_T ^*$ is discontinuous.
	\end{itemize}
\end{thm}

\begin{proof} Statement (i) follows from Lemma~\ref{l.index1} while (ii) follows by an inspection of the proof of Theorem~\ref{t.main}, replacing the use of the Kakutani-Rokhlin lemma with the hypothesis $N_T=\infty$. It remains to show (iii) which is the main content of the theorem in hand.

Let $T$ be an invertible measure preserving transformation such that $1 < N_T < \infty$.   We will show that $C_T ^*$ is discontinuous by proving that it has a jump discontinuity at $\alpha = \frac{2 N_T - 2}{2 N_T-1}$.   This will be done by showing that for every $\epsilon > 0$ we have that $C_T^*(\frac{2N_T-2}{2N_T - 1} - \epsilon) \geq \frac{N_T}{N_T-1}$ and subsequently showing that $C_T^\ast(\alpha) = 1$ for all $\frac{2N_T - 2}{2N_T - 1} < \alpha < 1$.

We now show that given $\epsilon > 0$, $C_T^*(\frac{2N_T-2}{2N_T - 1} - \epsilon) \geq \frac{N_T}{N_T-1}$.   By the definition of $N_T$, there exists a set $A \subset \Omega$ with $\mu(A) > 0$ such that $A, TA, \ldots, T^{N_T-1}A$ are pairwise disjoint.   Let $\tilde{A} \coloneqq \cap_{j=-\infty}^{\infty}T^{jN_T}A$. Note that $\mu(\tilde A ) = \mu(A)$.  Let now
\[
E \coloneqq T\tilde{A} \cup \cdots \cup T^{N_T - 1}\tilde{A}
\]
so that $\mu(E\cap \tilde A)=0$. Observe also that $T^{-1}\tilde A = T^{N_T-1}\tilde{A}$, $T^{-2}\tilde{A} = T^{N_T-2}\tilde{A}$, \ldots, $T^{-N_T+1}\tilde{A} = T\tilde{A}$.  So if $\omega \in \tilde{A}$
\[
\begin{split}
M^\ast \chi_E(\omega) &\geq \frac{1}{2(N_T - 1) + 1}\sum_{j= -(N_T - 1)} ^{N_T - 1}\chi_E(T^j \omega)
\\
& = \frac{1}{2(N_T - 1) + 1}\left[\sum_{j=1}^{N_T - 1}\chi_E(T^{-j} \omega) + \sum_{j=1}^{N_T - 1}\chi_E(T^j \omega)\right]
\\
& = \frac{1}{2(N_T - 1) + 1} \cdot 2(N_T - 1) = \frac{2N_T - 2}{2N_T - 1}.
\end{split}
\]
Since we obviously have that $M^\ast\chi_E(\omega) = 1$ on $E$ and $\mu(\tilde A\cap E)=0$, we conclude that
\[
\mu\Big(\Big\{\omega \in \Omega : M^\ast \chi_E(\omega) \geq \frac{2N_T - 2}{2N_T - 1}\Big\}\Big) \geq \mu(\tilde A)+\mu(  E)=\mu(\tilde{A}) N_T,
\]
where in the last equality we used that $\mu(E) = (N_T - 1)\mu(\tilde{A})$ since $T$ is measure preserving. Thus, for every $\epsilon>0$ we have
\[
C_{T}^\ast \Big(\frac{2N_T - 2}{2N_T - 1} - \epsilon\Big) \geq \frac{\mu(\tilde{A})\cdot N_T}{(N_T - 1)\mu(\tilde{A})} = \frac{N_T}{N_T - 1}.
\]

It remains to show that, if $\frac{2N_T - 2}{2 N_T - 1} < \alpha < 1$, then $C_T^\ast(\alpha) = 1$. For this let $E \subset \Omega$ with $0<\mu(E)<1$. It suffices to show that for $\mu$-a.e. $\omega \in \Omega \backslash E$ we have $M^\ast\chi_E(\omega) \geq \frac{2N_T - 2}{2N_T - 1}$.  To do this, it suffices to show that for any $J \leq 0 \leq K$ with $J \neq K$ and for $\mu$-a.e. $\omega \in \Omega$ we have
\[
\frac{1}{K - J + 1}\sum_{i=J}^{K}\chi_E(T^i\omega) \leq \frac{2N_T - 2}{2N_T - 1}.
\]
Note that, as $\omega \notin E$, if $K - J + 1 \leq 2N_T - 1$ then
\[
\frac{1}{K - J + 1}\sum_{i=J}^{K}\chi_E(T^i\omega) \leq \frac{(K - J + 1) - 1}{K - J + 1} \leq \frac{2N_T - 2}{2N_T - 1},
\]
so we may assume without loss of generality that $K - J + 1 > 2N_T - 1$.   Now we claim that for $\mu$-a.e. $\omega \in \Omega \setminus E$ at least one of $T\omega$, $T^2\omega$, \ldots, $T^{N_T}\omega$ lies in $\Omega \setminus E$.

To see this let us define $n_{\Omega \setminus E}(\omega)$ to be the return time of a point $\omega\in \Omega \setminus E$, namely $n_{\Omega \setminus E}(\omega)\coloneqq \inf\{n\geq 1:\, T^n\omega \in \Omega \setminus E\}$. By Poincar\'e recurrence we	 have that, $\mu$-a.e.
\[
\Omega \setminus E =\bigcup_{k=1} ^\infty \{\omega\in \Omega\setminus E:\, n_{\Omega \setminus E}(\omega)=k\}\eqqcolon \bigcup_{k=1} ^\infty \Lambda_k.
\]
Now if $\mu(\Lambda_k)>0$ for some $k$ we have that $k\leq N_T$. Indeed, if we had $k>N_T\Leftrightarrow k-1\geq N_T$ then we would have that $\Lambda_k,T(\Lambda_k),\ldots,T^{k-1}(\Lambda_k)$ are disjoint, contradicting the definition of $N_T$. Thus,
\[
\Omega\setminus E =\bigcup_{1\leq k\leq N_T} \Lambda_k \cup \mathcal O,
\]
where $\mu(\mathcal O)=0$. This means that for $\mu$-a.e. $\omega\in\Omega\setminus E$ we have that $\omega\in \Lambda_k$ for some $1\leq k\leq N_T$. Thus, for $\mu$-a.e. $\omega\in\Omega\setminus E$ there exists $1\leq k\leq N_T$ such that $T^k\omega \in \Omega\setminus E$, proving the claim.

Let us write $K-J+1 = N_T r + s $ where $r \geq 1$ and $0 \leq s \leq N_T - 1$.  For $\mu$-a.e. $\omega \in \Omega \backslash E$ we then have
\[
\frac{1}{K - J + 1}\sum_{i=J}^{K}\chi_E(T^i \omega) \leq \frac{(N_T-1)r + s}{N_T r + s}=1-\frac{1}{N_T+s/r}
\]
which is bounded above by $\frac{2N_T-2}{2N_T-1}$, seen by observing that the right hand side is bounded above by the value obtained by using that $s/r\leq N_T-1$.
\end{proof}

The following corollary is an immediate consequence of Theorem~\ref{t.index}.

\begin{cor}\label{c.index}Let $T$ be an invertible measure preserving transformation on a probability space $(\Omega,\Sigma,\mu)$ with index $N_T\in[1,\infty]$. Then $C_T ^*\in C^\infty(0,1)$ if and only $N_T = 1$ or $N_T = \infty$.	
\end{cor}

\section{One sided discrete and ergodic maximal operators}\label{s.onesided}
Due to its prevalence in ergodic theory, it is appropriate for us to briefly discuss the smoothness of Tauberian constants associated with \emph{one-sided} ergodic maximal operators.   Given a measure-preserving transformation $T$ on a  probability space $(\Omega, \Sigma, \mu)$, the associated one-sided ergodic maximal operator $T^{*+}$ is given by
\[
T^{*+}f(\omega) \coloneqq \sup_{N \geq 0} \frac{1}{N + 1}\sum_{j=0}^{N}|f(T^j\omega)|
\]
and the corresponding Tauberian constants $C^{*+}(\alpha)$ are given by
\[
C^{*+}_T(\alpha) \coloneqq \sup_{\substack{E \subset \Omega \\ \mu(E) > 0}} \frac{1}{\mu(E)}\mu(\{\omega \in \Omega :\, T^{*+}\chi_E(\omega) > \alpha\}).
\]

In general, $C^{*+}(\alpha)$ need not be H\"older continuous on $(0,1)$.  For example, we may
define $T$ on $[0,1)$ equipped with the Lebesgue measure by $T(x) \coloneqq \left(x + \frac{1}{2}\right) \bmod 1$.  The associated Tauberian constants $C^{*+}_T (\alpha)$ satisfy the formula

\[
C^{*+}(\alpha) =
\begin{cases}
  2 & \text{if}\ 0 < \alpha < \frac{1}{2} \\
  1          & \text{if}\ \frac{1}{2} \leq \alpha < 1.
\end{cases}
\]
and hence $C^{*+}_T (\alpha)$ is not continuous on $(0,1)$.  However, similarly to the two-sided case, if $T$ is a non-periodic transformation we have that $C^{*+}_T (\alpha)$ is smooth on $(0,1)$, in fact satisfying the formula
\[
C^{*+}_T (\alpha) = \frac{1}{\alpha}.
\]
Defining the one-sided discrete Hardy-Littlewood maximal operator
\[
\tilde{M}_{\textup{HL}}^+f(n) \coloneqq \sup_{N \geq 1}\frac{1}{N}\sum_{j=0}^{N-1}|f(n + j)|
\]
and its associated Tauberian constants
\[
\tilde{C}_{\textup{HL}}^+(\alpha) \coloneqq \sup_{\substack{E \subset \mathbb{Z} \\ 0 < \#E < \infty}} \frac{1}{\#E}\#\{n \in \mathbb{Z}:\, \tilde{M}_{\textup{HL}}^+\chi_E (n) > \alpha\},
\]
we also have $\tilde{C}_{\textup{HL}}^+(\alpha) = \frac{1}{\alpha}.$

\begin{thm}
Let $T$ be  a non-periodic transformation on the  probability space $(\Omega, \Sigma, \mu)$. Then  $C^{*+}_T (\alpha)$ is smooth on $(0,1)$, being given by the formula
\[
C^{*+}_T (\alpha) = \frac{1}{\alpha}.
\]
Moreover the Tauberian constants $\tilde{C}_{\textup{HL}}^+(\alpha)$ satisfy the formula
\[
\tilde{C}_{\textup{HL}}^+(\alpha) = \frac{1}{\alpha}.
\]
\end{thm}

\begin{proof}
By the proof of the Birkhoff Ergodic Theorem  (see, e.g. \cite{Pet}) we immediately realize that
\[
C^{*+}_T (\alpha) \leq 1/\alpha .
\]
The converse inequality $C^{*+}_T (\alpha) \geq 1/\alpha$ follows from the Kakutani Rokhlin Lemma and the observation that the Tauberian constants $\tilde{C}^+_{\textup{HL}}(\alpha)$ associated  with the discrete one-sided Hardy-Littlewood maximal operator $\tilde{M}_{\textup{HL}}^+$, defined by
\[
\tilde{M}_{\textup{HL}}^+f(n) \coloneqq \sup_{N \geq 1}\frac{1}{N}\sum_{j=0}^{N-1}|f(n + j)|,
\]
satisfy the equality
\[
\tilde{C}^+_{\textup{HL}}(\alpha) = \frac{1}{\alpha}.
\]
The latter may be seen to hold from the classical paper on maximal operators \cite{HL} by Hardy and Littlewood, and the details are left to the reader.
\end{proof}

\section{Future Directions}\label{s.problems}
The results in this paper suggest the following problems that the authors believe would be suitable avenues for further research.

\begin{problem}
We have shown that if $U_1, \ldots, U_n$ form a non-periodic collection of commuting  transformations on the non-atomic probability space $(\Omega, \Sigma, \mu)$, the associated Tauberian constants $C^\ast_{S}(\alpha)$ are H\"older continuous over any closed interval $K$ in $(0,1)$.  Must $C^\ast _{S}(\alpha) \in C^p(0,1)$ for every $p > 1$?   Must in fact $C^\ast_{S}(\alpha)$ be smooth in $(0,1)$?   We remark that the analogues of this problem for the discrete strong maximal operator $\tilde{M}_S$ and the continuous strong maximal operator $M_S$ remain unsolved as well.
\end{problem}

\begin{problem}
The use of the Katznelson-Weiss lemma in this paper requires the condition that $U_1 , \ldots, U_n$ be a non-periodic collection of \emph{commuting} invertible measure preserving transformations on $(\Omega, \Sigma, \mu)$.  It would be very interesting to know to what extent both the conclusions of the Katznelson-Weiss lemma and the results of this paper
hold in the context of non-periodic collections of non-commuting operators $U_1, \ldots, U_n$.
\end{problem}

\begin{problem}
We strongly suspect that an analogue of the connection between the index of an invertible measure preserving transformation $T$ and the continuity of the associated Tauberian constant function $C_T ^*$, provided by Theorem~\ref{t.index} and Corollary~\ref{c.index}, should also exist in the multiparameter setting.  It is unclear, however, what precisely should be the ``index'' associated with a non-periodic collection of commuting invertible measure preserving transformations, and techniques along the lines of the proof of Theorem \ref{t.index} are largely unavailable in the higher dimensional scenario.  This is reminiscent of difficulties that arise in the theory of differentiation of integrals in the multiparameter setting that do not exist in the one-parameter setting.  This is a subject of ongoing research.
\end{problem}



\begin{bibsection}
\begin{biblist}

	
\bib{bh}{article}{
   author={Beznosova, Oleksandra},
   author={Hagelstein, Paul Alton},
   title={Continuity of halo functions associated to homothecy invariant
   density bases},
   journal={Colloq. Math.},
   volume={134},
   date={2014},
   number={2},
   pages={235--243},
   issn={0010-1354},
   review={\MR{3194408}},
}


\bib{busemannfeller1934}{article}{
author = {Busemann, H.},
author = {Feller, W.},
journal = {Fundamenta Mathematicae},
number = {1},
pages = {226-256},
publisher = {Institute of Mathematics Polish Academy of Sciences},
title = {Zur Differentiation der Lebesgueschen Integrale},
url = {http://eudml.org/doc/212688},
volume = {22},
year = {1934},
}	


\bib{CorF}{article}{
   author={C\'ordoba, A.},
   author={Fefferman, R.},
   title={On the equivalence between the boundedness of certain classes of
   maximal and multiplier operators in Fourier analysis},
   journal={Proc. Nat. Acad. Sci. U.S.A.},
   volume={74},
   date={1977},
   number={2},
   pages={423--425},
   issn={0027-8424},
   review={\MR{0433117 (55 \#6096)}},
}


\bib{Gu}{article}{
   author={de Guzm{\'a}n, M.},
   title={Differentiation of integrals in ${\bf R}^{n}$},
   conference={
      title={Measure theory},
      address={Proc. Conf., Oberwolfach},
      date={1975},
   },
   book={
      publisher={Springer},
      place={Berlin},
   },
   date={1976},
   pages={181--185. Lecture Notes in Math., Vol. 541},
   review={\MR{0476978 (57 \#16523)}},
}



\bib{hlp}{article}{
   author={Hagelstein, Paul},
   author={Luque, Teresa},
   author={Parissis, Ioannis},
   title={Tauberian conditions, Muckenhoupt weights, and differentiation
   properties of weighted bases},
   journal={Trans. Amer. Math. Soc.},
   volume={367},
   date={2015},
   number={11},
   pages={7999--8032},
   issn={0002-9947},
   review={\MR{3391907}},
}

\bib{HPSolErg}{article}{
	author = {Hagelstein, Paul},
  author = {Parissis, Ioannis},
  title = {Solyanik estimates in ergodic theory},
 journal = {Colloq. Math.},
 volume={145},
 date={2016},
 pages={193--207},
}


\bib{HP}{article}{
  author = {Hagelstein, Paul},
  author = {Parissis, Ioannis},
  title = {Solyanik Estimates in Harmonic Analysis},
conference={
      title={Special Functions, Partial Differential Equations, and Harmonic Analysis},
   },
  date = {2014},
   book={
      series={Springer Proc. Math. Stat.},
      volume={108},
      publisher={Springer, Heidelberg},
   },
  journal = {Springer Proceedings in Mathematics \& Statistics},
  pages = {87--103},
}

\bib{HP2014Holder}{article}{
  author= {Hagelstein, Paul},
  author= {Parissis, Ioannis},
  title={Solyanik estimates and local H\"older continuity of halo functions of geometric maximal operators},
 journal={Adv. Math.},
 volume={285},
 date={2015},
 pages={434--453},
 review={\MR{3406505}}
   }


\bib{HP2}{article}{
			Author = {Hagelstein, Paul},
			Author = {Parissis, Ioannis},
			Title = {Weighted Solyanik estimates for the Hardy-Littlewood maximal operator and embedding of $A_\infty$ into $A_p$},
			journal={J. Geom. Anal.},
volume={26},
date={2016},
pages={924--946},
			review={\MR{3472823}}
}

\bib{HPS}{article}{

author = {Hagelstein, Paul},		
		author = {Parissis, Ioannis},
author = {Saari, Olli},
			Eprint = {1601.00938},
Title= {Sharp inequalities for one-sided Muckenhoupt weights}
journal={submitted for publication}}


\bib{HS}{article}{
   author={Hagelstein, Paul},
   author={Stokolos, Alexander},
   title={Tauberian conditions for geometric maximal operators},
   journal={Trans. Amer. Math. Soc.},
   volume={361},
   date={2009},
   number={6},
   pages={3031--3040},
   issn={0002-9947},
   review={\MR{2485416 (2010b:42023)}},
}


\bib{HL}{article}{
author={Hardy, G. H.},
author={Littlewood, J. E.},
title={A maximal theorem with function theoretic applications},
journal={Acta Math.},
volume={54},
date={1930},
pages={81--116}}

\bib{KW}{article}{
author={Katznelson, Y.},
author={Weiss, B.}
title={Commuting measure-preserving transformations},
journal={Israel J. Math.},
volume={12}
date={1972}
pages={161--173}
review={\MR{MR0316680}}
}




\bib{Pet}{book}{
author={Petersen, K.},
title={Ergodic Theory},
publisher={Cambridge University Press},
date={1983},
review={\MR{0833286 (87i:28002)}}
}

\bib{Solyanik}{article}{
   author={Solyanik, A. A.},
   title={On halo functions for differentiation bases},
   language={Russian, with Russian summary},
   journal={Mat. Zametki},
   volume={54},
   date={1993},
   number={6},
   pages={82--89, 160},
   issn={0025-567X},
   translation={
      journal={Math. Notes},
      volume={54},
      date={1993},
      number={5-6},
      pages={1241--1245 (1994)},
      issn={0001-4346},
   },
   review={\MR{1268374 (95g:42033)}}
}


	
\end{biblist}
\end{bibsection}
\end{document}